\documentclass{amsart}
\usepackage{amsmath,amssymb,amsfonts}
\usepackage{enumerate}
\usepackage{tikz-cd}
\usepackage{bm}
\usepackage[total={6.5in,8.75in},
top=1.2in, left=0.9in, includefoot]{geometry}

\newtheorem{theorem}{Theorem}

\newtheorem{lemma}[theorem]{Lemma}

\theoremstyle{definition}
\newtheorem{definition}[theorem]{Definition}
\newtheorem{example}[theorem]{Example}
\newtheorem{remark}[theorem]{Remark}

\newcommand{\Z}{\mathbb{Z}}
\newcommand{\F}{\mathbb{F}_2}
\newcommand{\G}{\mathbb{G}}
\newcommand{\Gop}{{\mathbb{G}^{\mathrm{op}}}}
\newcommand{\Set}{\mathrm{Set}}
\newcommand{\oCat}{\omega\mathrm{Cat}}
\newcommand{\chains}{\mathrm{C}_\bullet}
\newcommand{\coAlg}{\mathrm{coAlg}}
\newcommand{\id}{\mathrm{id}}
\newcommand{\Hom}{\mathrm{Hom}}

\begin{document}

\title{An algebraic representation of globular sets} 

\author{A. M. Medina-Mardones}             

\email{amedinam@nd.edu}

\address{Mathematics Department,
         University of Notre Dame,
         255 Hurley, Notre Dame, IN 46556,
         USA}
     

\begin{abstract}
We describe a fully faithful embedding of the category of (reflexive) globular sets into the category of counital cosymmetric $R$-coalgebras when $R$ is an integral domain. This embedding is a lift of the usual functor of $R$-chains and the extra structure consists of a derived form of cup coproduct. Additionally, we construct a functor from group-like counital cosymmetric \mbox{$R$-coalgebras} to $\omega$-categories and use it to connect two fundamental constructions associated to oriented simplices: Steenrod's cup-$i$ coproducts and Street's orientals. The first defines the square operations in the cohomology of spaces, the second, the nerve of higher-dimensional categories.
\end{abstract}

\maketitle

\section{Introduction}
Globular sets are presheaves over a category $\G$ whose objects are non-negative integers. They generalize directed graphs and constitute one of the major geometric shapes for higher category theory, providing models for strict and non-strict higher-dimensional categories when enriched with further structure. 

We depict the representable globular set $\G_n$ for small values of $n$: 
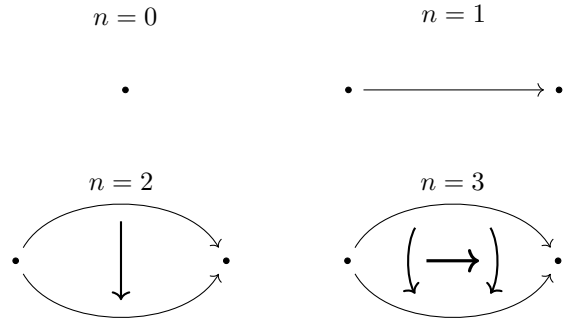
\begin{center}
	\begin{tikzpicture} [scale= .35]
	\draw node at (-4,0){ };
	\draw node at (4,0){ };
	\draw[fill] (0,0) circle  [radius=3pt];
	
	\draw node at (0,2.8){$n=0$} ;
	\draw node at (0,-2.4){} ;
	\end{tikzpicture}
	\qquad\qquad\!\!\!\!
	\begin{tikzpicture} [scale= .35]
	\draw[fill] (-4,0) circle  [radius=3pt];
	\draw[fill] (4,0) circle  [radius=3pt];
	
	\draw [->] (-4,0) [shorten >=0.2cm, shorten <=0.2cm,->] to (4,0) ;
	
	\draw node at (0,2.9){$n=1$} ;
	\draw node at (0,-2.4){} ;
	\end{tikzpicture}
	
	\begin{tikzpicture} [scale= .35]
	\draw[fill] (-4,0) circle  [radius=3pt];
	\draw[fill] (4,0) circle  [radius=3pt];
	
	\draw [->] (-4,0) [shorten >=0.2cm, shorten <=0.2cm,->, out=60,in=120] to (4,0) ;
	\draw [->] (-4,0) [shorten >=0.2cm, shorten <=0.2cm,->, out=-60,in=-120] to (4,0) ;
	
	\draw [->] (0,1.5) [thick] to (0,-1.5) ;
	
	\draw node at (0,3){$n=2$} ;
	\end{tikzpicture}
	\qquad\qquad
	\begin{tikzpicture} [scale= .35]
	\draw[fill] (-4,0) circle  [radius=3pt];
	\draw[fill] (4,0) circle  [radius=3pt];
	
	\draw [->] (-4,0) [shorten >=0.2cm, shorten <=0.2cm,->, out=60,in=120] to (4,0) ;
	\draw [->] (-4,0) [shorten >=0.2cm, shorten <=0.2cm,->, out=-60,in=-120] to (4,0) ;
	
	\draw [->] (-1,2) [thick,shorten >=0.3cm, shorten <=0.3cm,->, out=-120,in=120] to (-1,-2) ;
	\draw [->] (1,2) [thick,shorten >=0.3cm, shorten <=0.3cm,->, out=-60,in=60] to (1,-2) ;
	
	\draw [->] (-1,0) [very thick] to (1,0) ;
	
	\draw node at (0,3){$n=3$} ;
	\end{tikzpicture}
\end{center}
The globular set $\partial\, \G_{n+1}$ obtained by removing the identity from $\G_{n+1}$ models the $n$-sphere together with its antipodal map. We are interested in the functor $\chains$ of chains from globular sets to differential graded $R$-modules. Let $W$ be defined as the colimit of the diagram
\begin{equation*}
\chains(\partial\, \G_0) \to \chains(\partial\, \G_{1}) \to \cdots
\end{equation*}
induced from a standard set of inclusions $\G_n \to \G_{n+1}$. We notice that the antipodal map makes $W$ into a free differential graded $R[\Sigma_2]$-module. For any globular set $X$ we will construct a natural $R[\Sigma_2]$-module chain map
\begin{equation*}
\Delta : W \otimes \chains(X) \to \chains(X) \otimes \chains(X)
\end{equation*}
together with a natural chain map $\varepsilon : \chains(X) \to R$ satisfying appropriate counitality relations. We can think of this structure as a lift to the chain level of the counital cocommutative $R$-coalgebra on the homology of $X$ (a structure pre-dual to the usual cup product in cohomology).

We will show that when $R$ is an integral domain, this lift of the functor of chains is a fully faithful embedding of the category of globular sets into the category of counital cosymmetric $R$-coalgebras. We can think of this result as a non-linear globular form of the Dold-Kan Theorem. In more diagramatic language, our map fits into the following commutative diagram
\begin{equation*}
\begin{tikzcd}
\Set^\Gop \ar[d] \ar[r, dashed] \ar[rd, "\chains"]& \mathrm{\coAlg_R} \ar[d] \\
\mathcal{M}\mathrm{od}_R^{\,\Gop} \ar[r] & \mathrm{Ch}_{R}
\end{tikzcd}
\end{equation*}
where the lower triangle consists of a free functor followed by a fully faithful embedding and the upper triangle consists of a fully faithful embedding followed by a forgetful functor.

We will then focus on the full subcategory $\coAlg_{R}^{gl}$ of group-like counital cosymmetric $R$-coalgebras and on a model for strict higher-dimensional categories known as $\omega$-categories. We will describe a functor, similar to those used by Street, Brown, and Steiner in their respective studies of parity complexes, linear $\omega$-categories, and augmented directed complexes, from $\coAlg_{R}^{gl}$ to $\oCat$ behaving like a free functor on pasting diagrams. We will use our version to relate two fundamental constructions on oriented simplices: Steenrod's cup-$i$ coproducts and Street's orientals. The first defines the square operations on the cohomology of spaces, the second, the nerve of $\omega$-categories.

\subsection*{Acknowledgments} 
We would like to thank Chris Schommer-Pries, Tim Campion, Stephan Stolz, Dennis \mbox{Sullivan}, Richard Steiner, Manuel Rivera, Mahmoud Zeinalian, and the anonymous referee for their insights, questions, and comments about this project. 

\section{Globular sets and counital cosymmetric $R$-coalgebras}

In this section we will describe how to represent, when $R$ is an integral domain, the category of globular sets algebraically as a full subcategory of the category of counital cosymmetric $R$-coalgebras. These are models for counital $R$-coalgebras commutative up to coherent homotopies ($E_\infty$-coalgebras are examples). We will also review an important construction of Steenrod providing concrete examples of such $R$-coalgebra when $R = \F$ and used to define his square operations.

\subsection{Globular sets}

The \textbf{globe category} $\G$ has set of objects the non-negative integers and its morphisms are generated by
\begin{equation*}
\sigma_n, \tau_n : n \to n+1 \qquad \iota_n : n \to n-1 \quad \ \ \,
\end{equation*}
subject to the relations	
\begin{align} \label{equation: relations between generating morphisms}
\begin{aligned}
\tau_n \, \tau_{n-1} =\ & \sigma_n \, \tau_{n-1}   \\   \iota_{n+1} \, \tau_{n} =\ & \id_{n}
\end{aligned}
\quad & \quad
\begin{aligned}
\sigma_n  \, \sigma_{n-1} =\ & \tau_n \, \sigma_{n-1} \\   \iota_{n+1} \, \sigma_{n} =\ & \id_{n}.
\end{aligned}
\end{align}

Let $\Set$ be the category of small sets. We denote the category of contravariant functors from $\G$ to $\Set$ by $\Set^\Gop$ and refer to it as the category of \textbf{globular sets}. For a globular set $X$ we use the notation
\begin{equation*}
X_n = X(n) \qquad t_n = X(\tau_n) \qquad s_n = X(\sigma_n) \qquad i_n = X(\iota_n).
\end{equation*}
Furthermore, abusing notation, we let $t_n : X(k) \to X(n)$ stand for any composition of the form $t_n r$ where $r : X(k) \to X(n+1)$ is induced from an arbitrary morphism. Thanks to (\ref{equation: relations between generating morphisms}) this map is independent of $r$ and determined by the integer $k$. We follow a similar convention for $s_n$. 

\subsection{Augmented differential graded $R$-modules}

Let $R$ be a commutative and unital ring. The category of differential (homologically) graded $R$-modules concentrated in non-negative degrees is denoted $\mathrm{Ch}_R$. We reserve the word chain complex for when $R$ equals $\Z$.

Let $C$ be a differential graded $R$-module and $n$ a non-negative integer, we denote
\begin{equation*}
C_{\leq n} = C_0 \oplus C_1 \oplus \cdots \oplus C_n.
\end{equation*}

A pair $(C, \varepsilon)$ with $C$ and $\varepsilon : C \to R$ in $\mathrm{Ch}_R$ is called an \textbf{augmented differential graded $R$-module} and a morphism between two of them is a morphism of underlying differential graded $R$-modules making the diagram
\begin{center}
	\begin{tikzcd}
	C' \ar[rr] \ar[dr, "\varepsilon'"'] & \hspace*{-1cm} & C \ar[dl, "\varepsilon"] \\
	& R &
	\end{tikzcd}
\end{center}
commutative.

The functor $\chains : \Set^{\Gop} \to \mathrm{Ch}_R$ is defined for $X \in \Set^{\,\Gop}$ by
\begin{equation*}
\mathrm{C}_n(X) = R\{X_n\}\, \big/\, R\{i_{n}(X_{n-1})\} \qquad \partial_n = t_{n-1} - s_{n-1}\,.
\end{equation*}
It admits a natural lift to the category of augmented differential graded $R$-modules by defining for $x \in X_n$
\begin{equation*}
\varepsilon(x) = \begin{cases}
1 & n = 0 \\ 0 & n \neq 0.
\end{cases}
\end{equation*}

\subsection{Counital cosymmetric $R$-coalgebras}

Let $\Sigma_2$ be the group with one non-identity element $T$. Let us consider the following resolution of $R$ by free $R[\Sigma_2]$-modules: 
\begin{equation*}
W \quad = \quad R[\Sigma_2] \stackrel{\ 1-T}{\longleftarrow} R[\Sigma_2] \stackrel{\ 1+T}{\longleftarrow} R[\Sigma_2] \stackrel{\ 1-T}{\longleftarrow} \cdots
\end{equation*}
and let $\varepsilon_W : W \to R$ be the unique $R[\Sigma_2]$-linear map extending the identity $R \to R$.

Given any differential graded $R$-module $C$ we make $C \otimes C$ into a differential graded $R[\Sigma_2]$-module using the transposition of factors $T(x \otimes y) = (-1)^{rs} y \otimes x$ where $r$ and $s$ are the degrees of $x$ and $y$.

A \textbf{counital cosymmetric $R$-coalgebra} is an augmented differential graded \mbox{$R$-module} $(C, \varepsilon)$ together with 
\begin{equation*}
\Delta : W \otimes C \to C \otimes C
\end{equation*}
an $R[\Sigma_2]$-linear chain map making the following diagrams commute:
\begin{center}
	\begin{tikzcd}
	W \otimes C \ar[dr, "\varepsilon_W \otimes \id"']\ar[r, "\Delta"] & C \otimes C \ar[d, "1 \otimes \varepsilon"]\\ & C
	\end{tikzcd}
	\qquad \qquad
	\begin{tikzcd}
	W \otimes C \ar[dr, "\varepsilon_W \otimes \id"']\ar[r, "\Delta"] & C \otimes C \ar[d, "\varepsilon \otimes 1"]\\ & C\ .
	\end{tikzcd}
\end{center} 

A \textbf{coalgebra map} between counital cosymmetric $R$-coalgebras is a map $f$ of underlying augmented differential graded $R$-modules making the following diagram commute:

\begin{center}
	\begin{tikzcd}
	W \otimes C' \ar[d, "\Delta'"'] \ar[r, "\id \otimes f"] & W \otimes C \ar[d, "\Delta"]\\
	C' \otimes C' \ar[r, "f \otimes f"] & C \otimes C\, .
	\end{tikzcd}
\end{center}  

We denote the category of counital cosymmetric $R$-coalgebras with coalgebra maps by $\coAlg_R$.	

We use the adjunction isomorphism 
\begin{equation*}
\Hom_{R[\Sigma_2]}(W \otimes C , C \otimes C) \to \Hom_{R[\Sigma_2]} \big( W, \Hom(C, C \otimes C) \big)
\end{equation*}
to represent $\Delta$ by a collection of maps $\Delta_k : C \to C \otimes C$ satisfying
\begin{equation} \label{equation: Delta map in terms of Delta_k maps}
\partial \Delta_k - (-1)^{k} \Delta_k \partial = \big(1 + (-1)^k T \big)\Delta_{k-1}
\end{equation} 
with the convention that $\Delta_{-1} = 0$.

\subsection{Steenrod cup-$i$ coalgebras}

Alexander-Whitney's approximation to the diagonal map
\begin{equation*}
\Delta_0 : \chains \to \chains \otimes \chains
\end{equation*}
defines a natural non-cocommutative coproduct on the integral chains of any simplicial set whose linear dual descends to the commutative cup product on its cohomology. In \cite{steenrod1947products}, Steenrod constructed a cosymmetric $\Z$-coalgebra 
\begin{equation*}
\Delta : W \otimes \chains \to \chains \otimes \chains
\end{equation*} 
extending the Alexander-Whitney coproduct, which when considered with $\F$-coefficients defines the square operations
\begin{equation*}
Sq^k : H^\bullet(- ; \F) \to H^{\bullet + k}(- ; \F).
\end{equation*}

Since these operations are homological in nature, any pair of natural homotopy equivalent cosymmetric $\F$-coalgebra structures give rise to isomorphic square operations. Yet, Steenrod's original construction appears ubiquitously in the literature in various equivalent forms. For example, in \cite{medina2018algebraic}, the author finds it in the action of a finitely presented prop arising from just three maps: Alexander-Whitney's diagonal, the augmentation, and the join map. In \cite{medina2018cellular}, it is induced from the action of a cellular $E_\infty$-operad on the geometric realization of cubical sets. And in \cite{mcclure2003multivariable} and \cite{berger2004combinatorial},	McClure-Smith and Berger-Fresse find it in the action of their respective Sequence and Barratt-Eccles operads. 

The universality of this cosymmetric $\F$-coalgebra is formalized via an axiomatic characterization in \cite{medina2018axiomatic}. In this note, we provide further evidence for its fundamental nature by deriving from it in Theorem \ref{theorem steenrod and street} another fundamental construction: the nerve of higher-dimensional categories. 

Let us review its description as presented in \cite{medina2018axiomatic}. Let $P{n \choose k}$ be the set of all $U = \{0 \leq u_1 < \dots < u_k \leq n\}$. For any such $U$ define the composition of face maps 
\begin{equation*}
d_U = d_{u_1} \cdots \, d_{u_k}
\end{equation*}
and the pair
\begin{equation*}
\begin{split}
U^- &= \{ u_i \in U\, :\, u_i  \not\equiv i \text{ mod } 2 \} \\
U^+ &= \{ u_i \in U\, :\, u_i  \equiv i \text{ mod } 2 \}.
\end{split}
\end{equation*}

\begin{definition} (\cite{medina2018axiomatic}) \label{definition: steenrod coalgebra}
	For any simplicial set $X$ its \textbf{Steenrod cup-$i$ coalgebra} $\big( \chains(X; \F), \Delta, \varepsilon \big)$ is defined by
	\begin{equation} \label{equation: steenrod diagonal}
	\Delta_i (x)\ =
	\sum_{U \in P{n \choose n-i}} d_{U^-}\, x \otimes d_{U^+}\,x
	\end{equation}
	and 
	\begin{equation*}
	\varepsilon(x) = \begin{cases}
	1 & n = 0 \\ 0 & n \neq 0
	\end{cases}
	\end{equation*}
	where $x \in X_n$.
\end{definition}

\begin{remark}
	The cup product and Steenrod squares in cohomology are obtained from the Steenrod cup-$i$ coalgebra by defining
	\begin{equation*}
	[\alpha] \smallsmile [\beta] = \big[ (\alpha \otimes \beta) \Delta_0 \big]
	\end{equation*}
	and
	\begin{equation*}
	Sq^k[\alpha] = \big[ (\alpha \otimes \alpha) \Delta_{|\alpha| - k} \big]. 
	\end{equation*}
	The relationship between these two cohomological structures is known as the Cartan Formula. In \cite{medina2019effective}, using the definitions above, the author gave an effective chain level proof of the Cartan Formula, and in  \cite{medina2018persistence},  based on (\ref{definition: steenrod coalgebra}),  a novel algorithm for the computation of Steenrod squares of finite simplicial complexes was developed and added to the field of topological data analysis.
\end{remark}

\subsection{Globular $R$-coalgebras}

We now describe a counital cosymmetric $R$-coalgebra naturally associated to a globular set and state our main theorem.

\begin{definition} \label{definition: globular coalgebra}
	For any globular set  $X$ its \textbf{globular $R$-coalgebra} $\big( \chains(X; R), \Delta, \varepsilon \big)$ is defined by
	\begin{equation*}
	\Delta_k(x) =
	\begin{cases}
	0 & n < k \\ x \otimes x & n = k \\ 
	t_k x \otimes x \, + (-1)^{(n+1)k}\, x \otimes s_k x & k < n
	\end{cases}
	\end{equation*}
	and
	\begin{equation*}
	\varepsilon(x) =
	\begin{cases}
	1 & n = 0 \\ 0 & n > 0.
	\end{cases}
	\end{equation*}
	where $x \in X_n$.
\end{definition}

\begin{theorem} \label{theorem: dold-kan correspondence for globular sets}
	Let $R$ be an integral domain. The assignment 
	\begin{equation*}
	X \to (\chains(X; R), \Delta, \varepsilon)
	\end{equation*}  		
	induces a full and faithful embedding of $\Set^\Gop$ into $\coAlg_{R}$.
\end{theorem}

The proof of this theorem occupies Section 4.

\begin{remark}
	We can think of this statement as a non-linear globular form of the Dold-Kan Theorem. A conjecture, verified in the author's thesis \cite{medina2015thesis} for special cases, is that including the higher arity parts of an $E_\infty$-coalgebra structure on the chains of simplicial sets results in a similar non-linear (simplicial) Dold-Kan Theorem.
\end{remark}

\section{Group-like coalgebras and higher-dimensional categories}

\subsection{$\omega$-categories and the functor $\mu$}

In this subsection we recall the definition of $\omega$-categories, which are a globular model of strict higher-dimensional categories. We also review a natural construction associating an $\omega$-category to any differential graded $R$-module.  

\begin{definition} \label{definition: omega-category}
	An \textbf{$\omega$-category} is a globular set $X$ together with maps
	\begin{equation*}
	\circ_m : X_n \times_{X_m} X_n \to X_m
	\end{equation*}
	where
	\begin{equation*}
	X_n \times_{X_m} X_n = \big\{ (y,x) \in X \times X\ |\ s(y) = t(x) \big\}
	\end{equation*}
	satisfying relations of associativity, unitality and interchange. For the complete list of relations we refer the reader to Definition 1.4.8 in \cite{leinster2004higher}.	
	When $t_m(x) = s_m(y) = z$ we write $y \circ_z x$ for $y \circ_m x$. 
\end{definition}

The next definition appears in \cite{steiner2004omega} where is partially credited to \cite{brown2003cubical} and \cite{street1991parity}.
\begin{definition} [Street, Brown-Higgins, Steiner] \label{definition: mu functor}
	The functor
	\begin{equation*}
	\mu : \mathrm{Ch}_R \to \oCat 
	\end{equation*}
	is defined as follows: for $C$ a differential graded $R$-module let $\mu(C)$ be the $R$-submodule of the infinite product of $C$ with itself generated by all sequences
	\begin{equation*}
	c = (c_0^-, c_0 ^+, c_1^-, c_1^+, \dots)
	\end{equation*}  
	satisfying
	\begin{enumerate}[i)]
		\item $c_n^-, c_n^+ \in C_n,$
		\item $c_n^-, c_n^+ = 0 \text{ for } n >> 0,$
		\item $\partial c_{n+1}^- = \partial c_{n+1}^+ = c_n^+ - c_n^-.$
	\end{enumerate} 
	We can make this $R$-module into a globular set by defining
	\begin{equation*}
	\mu(C)_n = \{c \in \mu(C)\ :\ \forall k > n,\ c^-_k = c^+_k = 0 \}
	\end{equation*}
	and
	\begin{equation*}
	\begin{split}
	s_k(c) = &\ (c_0^-, c_0 ^+, \dots, c_{k-1}^-, c_{k-1}^+, c^-_k, c_k^-, 0, 0, \dots) \\
	t_k(c) = &\ (c_0^-, c_0 ^+, \dots, c_{k-1}^-, c_{k-1}^+, c^+_k, c_k^+, 0, 0, \dots) \\
	i_k(c) = &\ c.
	\end{split}
	\end{equation*}
	We can make this globular set into an $\omega$-category by defining 
	\begin{equation*}
	\begin{split}
	b\ \circ_c\ a =\ & b + a - c \\
	=\ & \big(b_0^- + a_0^- - c_0^-,\ b_0^+ + a_0^+ - c_0^+, \dotso \big).
	\end{split}
	\end{equation*}
\end{definition}

\subsection{Group-like coalgebras and the functor $\xi$}

In this subsection we define group-like elements in counital cosymmetric $R$-coalgebras and consider $\coAlg_{R}^{gl}$, the full subcategory of counital cosymmetric $R$-coalgebras admitting a basis of group-like elements. We then introduce a functor from $\coAlg_{R}^{gl}$ to $\oCat$ using the notion of atom associated to a group-like element.

\begin{definition} \label{definition: group-like element}
	Let $\big( C, \Delta , \varepsilon\big)$ be a counital cosymmetric coalgebra. We call $c \in C_n$ a \textbf{group-like element} if for any integer $k$ we have	
	\begin{equation*}
	\Delta_k (c) \in C_{\leq n} \otimes C_{\leq n}
	\end{equation*}
	\begin{equation*}
	\Delta_n(c) = c \otimes c
	\end{equation*} 
	and, when $n = 0$,
	\begin{equation*}
	\varepsilon(c) = 1.
	\end{equation*}  
	
	We say that $C$ is \textbf{group-like} if it admits a basis of group-like elements and denote the full subcategory of group-like counital cosymmetric $R$-coalgebras as $\coAlg^{gl}_R$.
\end{definition}

A consequence of the following lemma applied to the identity map is that if a counital cosymmetric coalgebra admits a basis of group-like elements, then that basis is unique. 

\begin{lemma} \label{lemma: coalgebra maps take group-like to group-like or 0}
	Let $R$ be an integral domain. If $f : R[A] \to R[B]$ is a coalgebra map between counital cosymmetric $R$-coalgebras with bases of group-like elements $A$ and $B$. Then, for any $a \in A$ either $f(a) = 0$ or there exists $b \in B$ such that $f(a) = b$.  
\end{lemma}

\begin{proof}
	For $a \in A_n$ there is a collection of elements $b_i \in B_n$ and coefficients $\beta_i \in R$ such that
	\begin{equation} \label{equation: pancito}
	f(a) = \sum_i \beta_{i}\,b_i.
	\end{equation}
	Applying $\Delta_{n}$ to (\ref{equation: pancito}) gives
	\begin{equation*}
	\sum_i \beta_{i}\, b_{i} \otimes b_{i} = \Delta_n f(a) = (f \otimes f) \Delta_n (a)  = \sum_{i,\,j} \beta_{i}\, \beta_{j}\ b_i \otimes b_j.
	\end{equation*}
	The equations $0 = \beta_{i}\, \beta_{j}$ for $i \neq j$ together with $\beta_{i} = \beta_{i}^2$ imply, since $R$ is an integral domain, that each coefficient $\beta_{i}$ equals $0$ except possibly one of them that must \mbox{equal $1$}.
\end{proof}

\begin{example} 
	Steenrod cup-$i$ coalgebras as well as globular $R$-coalgebras are group-like.
\end{example}

\begin{definition}
	Let $C$ be a differential graded $R$-module with a basis $B$. For $b \in B$ let $\pi_b : C \to R$ be the $R$-linear map sending $b$ to $1$ and the other basis elements to $0$. Define the maps
	\begin{equation*}
	\pi_b^+, \, \pi_b^- : C \otimes C \to C
	\end{equation*} 
	by
	\begin{equation*}
	\pi_b^+ = \id \otimes \pi_b \qquad \text{and} \qquad \pi_b^- = \pi_b^+ T
	\end{equation*}
	where $T$ is the transposition of factors.
\end{definition}

We make a note of the following straightforward observation for later use:
\begin{lemma} \label{lemma: boundary and projection for specific degrees}
	Let $C$ be a differential graded $R$-module with a basis $B$. If $b \in B_n$ and $\eta \in \{+,-\}$ then
	\begin{equation*}
	\partial \pi^\eta_b = \pi^\eta_b \partial
	\end{equation*} 
	on $C_{\leq n} \otimes C_{\leq n}$ and
	\begin{equation*}
	\pi_b^\eta = 0 
	\end{equation*}
	on $C_{\leq n-1} \otimes C_{\leq n-1}$.
\end{lemma}

\begin{definition} \label{definition: coalgebra atom}
	Let $(C, \Delta, \varepsilon)$ be a counital cosymmetric coalgebra. For every group-like element $b \in C$ define its \textbf{atom} as
	\begin{equation*}
	\langle b \rangle = \big( \langle b \rangle_0^-, \langle b \rangle_0^+, \langle b \rangle_1^-, \langle b \rangle_1^+,\, \dots \big)
	\end{equation*}
	with 
	\begin{equation*}
	\langle b \rangle_k^\eta = 
	\begin{cases} 
	(-1)^k\, \pi^-_b \Delta_k b & \eta = - \\
	\pi^+_b \Delta_k b & \eta = +.
	\end{cases}
	\end{equation*}
\end{definition}

\begin{lemma}
	Let $C$ be a counital cosymmetric coalgebra. For any group-like element $b \in C$ the atom $\langle b \rangle$ is in $\mu(C)$. 
\end{lemma}

\begin{proof}
	We need to prove that for any group-like element $b$ of degree $n$ the sequence $\langle b \rangle$ satisfies conditions i), ii), and iii) in Definition \ref{definition: mu functor}. Since the first two conditions are immediate, we are left with showing that for any non-negative integer $k$ 
	\begin{equation} \label{equation: tostadas}
	\partial \langle b \rangle_{k+1}^+ = \partial \langle b \rangle_{k+1}^- = \langle b \rangle_{k}^+ - \langle b \rangle^-_{k.}
	\end{equation}
	For $k > n$ and $\eta \in \{+ , -\}$ we have $\langle b \rangle_{k}^{\eta} = 0$ so (\ref{equation: tostadas}) holds. 
	For $k = n$ we notice that $\langle b \rangle_{k}^+ = \langle b \rangle_{k}^- = b$ and (\ref{equation: tostadas}) follows. For $k < n $, using Lemma \ref{lemma: boundary and projection for specific degrees}, we have 
	\begin{equation*}
	\begin{split}
	\langle b \rangle_{k}^+ - \langle b \rangle_{k}^- =\ & \pi_b^+ \Delta_{k}b - (-1)^{k} \pi_b^- \Delta_{k}b = \pi_b^+ \big( 1+(-1)^{k+1} T \big ) \Delta_{k} b \\ =\ & 
	\pi^+_b \big(\partial \Delta_{k+1} b - (-1)^k \Delta_{k+1} \partial b \big) = \partial \pi^+_b \Delta_{k+1} b \\ =\ &  
	\partial \langle b \rangle_{k+1}^+
	\end{split}
	\end{equation*}
	and 
	\begin{equation*}
	\begin{split}
	\langle b \rangle_{k}^+ - \langle b \rangle_{k}^- =\ & \pi_b^+ \Delta_{k}b - (-1)^{k} \pi_b^- \Delta_{k}b = \pi_b^- \big( T + (-1)^{k+1} \big) \Delta_{k} b \\ =\ & 
	(-1)^{k+1} \pi^-_b \big( \partial \Delta_{k+1} b - (-1)^k \Delta_{k+1} \partial b \big) \\ =\ &  (-1)^{k+1} \partial \pi^-_b (\Delta_{k+1} b) = \partial \langle b \rangle_{k+1}^-
	\end{split}
	\end{equation*}
	as desired.
\end{proof}

\begin{lemma} \label{lemma: the functor xi is well defined}
	Let $R$ be an integral domain. The assignment sending a group-like counital cosymmetric $R$-coalgebra $C$ to the sub-$\omega$-category of $\mu(C)$ generated by its atoms is functorial.
\end{lemma}

\begin{proof}
	The statement follows from the fact, proven in Lemma \ref{lemma: coalgebra maps take group-like to group-like or 0}, that when $R$ is an integral domain a coalgebra map between group-like $R$-coalgebras sends group-like elements to either group-like elements or \mbox{to $0$}. 
\end{proof}

\begin{definition}
	Let 
	\begin{equation*}
	\xi : \coAlg_R^{gl} \to \oCat
	\end{equation*}
	be the functor described in Lemma \ref{lemma: the functor xi is well defined}.
\end{definition}

\subsection{Street's orientals}

In this subsection we state the second main result of this note: the functor $\xi$ sends the Steenrod coalgebra of a standard simplex to the free $\omega$-category generated by that simplex. 

Historically, Roberts \cite{roberts1977mathematical} pioneered the idea of using higher-dimensional categories as the coefficient objects for non-abelian cohomology. A key ingredient for this enterprise is the construction of a nerve functor from $\omega$-categories to simplicial sets. Such a functor $N$ can be obtained from the construction of a natural cosimplicial $\omega$-category
\begin{equation*}
\begin{split}
\mathcal{O} :\ \bm{\Delta}\hspace*{1pt} &\to \oCat \\
[n] &\mapsto \,  \mathcal{O}_n
\end{split}
\end{equation*}
by setting
\begin{equation*}
N(\mathcal{C})_n = \Hom_{\oCat} \big( \mathcal{O}_n,\, \mathcal{C} \big).
\end{equation*}
This was accomplished by Street in \cite{street1987algebra} where he says the following about the $\omega$-categories $\mathcal{O}_n$\,: ``[t]hese objects seem to be fundamental structures of nature so I decided they should have a short descriptive name. I settled on oriental."

We will not use the original definition of Street but an equivalent one given by Steiner in \cite{steiner2004omega} and further explored in \cite{steiner2007orientals}. It is presented as Definition \ref{definition: oriental} after a review of Steiner's theory of augmented directed complexes. 

We are ready to state the second main result of this work. 
\begin{theorem} \label{theorem steenrod and street}
	Let $\big( \chains(\bm{\Delta}^n; \F), \Delta, \varepsilon \big)$ be the Steenrod cup-$i$ coalgebra associated to the $n$-th representable simplicial set $\bm{\Delta}^n$. Then,
	\begin{equation*}
	\xi \big( \chains(\bm{\Delta}^n; \F), \Delta, \varepsilon \big) = \mathcal{O}_n.
	\end{equation*} 
\end{theorem}

The proof of this theorem occupies subsection \ref{subsection: proof of theorem Steenrod and street}. 

\subsection{Steiner's augmented directed complexes} \label{subsection: Steiner's ADCs}

In this section we give an extremely abridged exposition of Steiner's rich theory of augmented directed complexes with the aim of proving Theorem \ref{theorem steenrod and street}. The original source is \cite{steiner2004omega}.

We refer to the objects of $\mathrm{Ch}_{\Z}$ simply as chain complexes.

Let $C$ be a chain complex together with a basis. We write $C^+$ for the submonoid containing all elements written as linear combinations of basis elements with only non-negative coefficients. We use the following notation for the induced canonical decomposition:
\begin{equation*}
c = c^+ - c^-
\end{equation*}
with $c^+$ and $c^-$ in $C^+$.

Let $(C, \varepsilon)$ be an augmented chain complex with a basis $B$. For $b \in B_n$ define recursively
\begin{equation} \label{equation: steiner atom}
b_i^+ = 
\begin{cases}
0 & i > n \\
b & i = n \\
(\partial b_{i+1}^+)^+ & i < n
\end{cases}
\qquad \text{and} \qquad
b_i^- = 
\begin{cases}
0 & i > n \\
b & i = n \\
(\partial b_{i+1}^-)^- & i < n.
\end{cases}
\end{equation}
The basis is said to be \textbf{unital} if $\varepsilon(b_0^+) = \varepsilon(b_0^-) = 1$ for every $b \in B$.

\begin{definition}
	A \textbf{strong augmented directed complex} or simply a \textbf{SADC} is an augmented chain complex $C$ with a unital basis such that the transitive closure of the reflexive relation $\leq$ defined by 
	\begin{equation*}
	c_1 \leq c_2
	\end{equation*}
	if and only if
	\begin{equation*}
	(\partial c_2)^- \! - c_1 \in C^+
	\end{equation*}
	or
	\begin{equation*}
	(\partial c_1)^+ \! - c_2 \in C^+
	\end{equation*}
	is anti-symmetric, i.e., it defines a partial order on $C$.
	
	A morphism between two SADCs is an augmented chain map $f : C_1 \to C_2$ such that
	\begin{equation*}
	f(C_1^+) \subset f(C_2^+).
	\end{equation*} 
\end{definition}

\begin{definition} \label{definition: steiner atom}
	Let $(C, B)$ be a SADC. For $b \in B$ define its \textbf{Steiner atom} to be
	\begin{equation*}
	\big(b_0^-, b_0^+, b_1^-, b_1^+, \dots \big) \in \mu(C).
	\end{equation*}
\end{definition}

Steiner showed that assigning to a SADC, let us call it $(C,B)$, the sub-$\omega$-category generated inside $\mu(C)$ by its Steiner atoms defines a full and faithful embedding 
\begin{equation*}
\nu : \mathrm{SADC} \to \oCat
\end{equation*}
independent of the commutative and unital ring $R$ used in the construction of $\mu$. We refer the reader to 5.6, 6.1, and 6.2 in \cite{steiner2004omega} for these statements. 

Additionally, Steiner gives the following definition of Street's orientals in 3.8 loc. cit.:

\begin{definition} [\cite{steiner2004omega}] \label{definition: oriental}
	Let $\bm{\Delta}^n$ denote the $n$-th representable simplicial set. The chain complex $\chains(\bm{\Delta}^n; \Z)$ together with the canonical basis
	\begin{equation*}
	B = \big\{ [m] \to [n]\, :\, \text{injective} \big\}
	\end{equation*}
	define a SADC and 
	\begin{equation*}
	\mathcal{O}_n = \nu \big( \chains(\bm{\Delta}^n; \Z),\ B \big).
	\end{equation*}
\end{definition}  

\subsection{Proof of Theorem \ref{theorem steenrod and street}} \label{subsection: proof of theorem Steenrod and street}

We will exhibit a bijection between the set of atoms of $\xi \big( \chains(\bm{\Delta}^n; \F), \Delta, \varepsilon \big)$ and of Steiner atoms of $\nu \big( \chains(\bm{\Delta}^n; \Z), B \big)$ which, since these are generators, will establish the theorem. 

We will verify that for every non-degenerate simplex $\sigma : [m] \to [n]$ and $\eta \in \{-,+\}$ we have    
\begin{equation} \label{equation: in simplex atoms equals steiner atoms}
\sigma^\eta_i = \pi_\sigma^\eta \Delta_i \sigma
\end{equation}
where this equality holds with $\Z$-coefficients using the canonical set lift $\F \to \Z$ with $0 \mapsto 0$ and $1 \mapsto 1$.

For $i > m$, both sides of (\ref{equation: in simplex atoms equals steiner atoms}) are equal to $0$.

For $i \leq m$, let $r = m-i$. Then, by (\ref{equation: steenrod diagonal}), we have
\begin{equation*}
\pi_\sigma^-\Delta_i \sigma = \sum_{\substack{U \in P{m \choose r} \\ U^- =\ \emptyset}} d_{U^+}\sigma
\qquad \text{and} \qquad
\pi_\sigma^+\Delta_i \sigma = \sum_{\substack{U \in P{m \choose r} \\ U^+ =\ \emptyset}} d_{U^-}\sigma.		
\end{equation*}
Then, in the case $r = 0$,  (\ref{equation: in simplex atoms equals steiner atoms}) holds because $\pi^-_\sigma \Delta_i \sigma = \sigma$ and $\pi^+_\sigma\Delta_i \sigma = \sigma$, so $\pi^-_\sigma \Delta_i \sigma = \sigma_i^-$ and $\pi^+_\sigma\Delta_i \sigma = \sigma_i^+$. Assuming the identity for $r$ we compute 
\begin{equation*}
\begin{split}
\partial \sigma_i^- = 
\sum_j (-1)^j d_j \sigma_i^- = 
\sum_j (-1)^j \sum_{\substack{U \in P{m \choose r} \\ U^- =\ \emptyset}} d_jd_{U^+}\sigma.
\end{split}
\end{equation*}
We will prove the identity for $r+1$ by rewriting the above identity as
\begin{equation*}
\partial \sigma_i^- = \sum_{\substack{U \in P{m \choose r+1} \\ U^- =\ \emptyset}} d_{U^+}\sigma 
\quad - \quad
\sum_{\substack{U \in P{m \choose r+1} \\ U^+ =\ \emptyset}} d_{U^-}\sigma.
\end{equation*}
For $U = \{u_1 < \dots < u_r\} \in P{n \choose r}$ with $U^- =\, \emptyset$ and $0 \leq j \leq i$ we can use the simplicial identities to write 
\begin{equation*}
d_jd_{U^+} = d_j d_{u_1} \dots \, d_{u_r} = d_{u_1} \dots d_{u_l}\, d_{j+l}\,  d_{u_l+1} \dots\, d_{u_r}
\end{equation*}
with $u_l < j+l < u_{l+1}$. Notice that if $j \equiv 1$ mod 2 and $l < r$ then 
\begin{equation*}
V = \{u_1 < \dots < u_l < j+l < \widehat{u}_{l+1} < \dots < u_r\} \in \textstyle{P{m \choose r}}
\end{equation*}
with $V^- = \emptyset$ and, calling $k = u_{l+1} - l - 1$, 
\begin{equation*}
(-1)^j d_j d_U + (-1)^{k} d_{k} d_V = 0.
\end{equation*}
If $j \equiv 0$ mod 2 and $1 < l$ then 
\begin{equation*}
W = \{u_1 < \dots < \widehat{u}_l < j+l < u_{l+1} < \dots < u_r\} \in \textstyle{P{m \choose r}}
\end{equation*}
with $W^- = \emptyset$ and, calling $k = u_l - l$,
\begin{equation*}
(-1)^j d_j d_U + (-1)^{k} d_{k} d_W = 0.
\end{equation*}
This implies that the only non-zero terms are of the form
\begin{equation*}
\begin{cases}
d_{u_1} \dots d_{u_r} d_{j+r} & j \text{ odd}\\
d_{j} d_{u_1} \dots d_{u_r}   & j \text{ even}
\end{cases}
\end{equation*}
for $U = \{u_1 < \dots < u_r\} \in P{n \choose r}$ with $U^- =\, \emptyset$. Therefore,
\begin{equation*}
\partial \sigma_{i} = 
\sum_{\substack{U \in P{m \choose r+1} \\ U^- =\ \emptyset}} d_{U^+}\sigma 
\quad - \quad
\sum_{\substack{U \in P{m \choose r+1} \\ U^+ =\ \emptyset}} d_{U^-}\sigma
\end{equation*}
as claimed.

\section{Proof of Theorem \ref{theorem: dold-kan correspondence for globular sets}}

We will prove Theorem \ref{theorem: dold-kan correspondence for globular sets} by establishing a sequence of lemmas. Unless stated otherwise, all algebraic constructions are taken over a general commutative and unital ring $R$. 

\begin{lemma} \label{lemma: lift well defined for objects}
	For any globular set $X$ the triple $\big( \chains(X), \Delta, \varepsilon \big)$ is a counital cosymmetric $R$-coalgebras.
\end{lemma}

\begin{proof}
	Showing that $\Delta : W \otimes \chains(X) \to \chains(X) \otimes \chains(X)$ is a $R[\Sigma_2]$-linear chain map is equivalent to establishing (\ref{equation: Delta map in terms of Delta_k maps}) for all $k \geq 0$. We will split the verification into six cases. For the remainder of this proof let us consider $x \in X_n$.  \par\ \\
	If $k = n = 0$ :
	\begin{equation*}
	\begin{split}
	\partial \Delta_k x - (-1)^{k} \Delta_k \partial x = \,& 
	\partial (x \otimes x) \\ =\, & 
	0.
	\end{split}
	\end{equation*}
	If $k = 0 < n$ :
	\begin{equation*}
	\begin{split}
	\partial \Delta_k x - (-1)^{k} \Delta_k \partial x = \,& 
	t_0 x \otimes (t_{n-1} - s_{n-1}) x \ + \ (t_{n-1} - s_{n-1}) x \otimes s_0 x \\ - \, &
	t_0 t_{n-1} x \otimes t_{n-1} x \ - \ t_{n-1} x \otimes s_0 t_{n-1} x \\ +\, &
	t_0 s_{n-1} x \otimes s_{n-1} x \ + \ t_{n-1} x \otimes s_0 s_{n-1} x \\ =\, &
	0.
	\end{split}
	\end{equation*}
		If $0 < k = n+1$ :
	\begin{equation*} 
	\begin{split}
	\partial \Delta_k x - (-1)^k \Delta_k \partial x\, =\, & 
	0 \\ =\, &
	x \otimes x - x \otimes x \\ =\, &
	\big( 1 + (-1)^{n+1} T \big)(x \otimes x)  \\ =\, &
	\big( 1 + (-1)^{k} T \big) \Delta_{k-1}(x).
	\end{split}
	\end{equation*}
	If $0 < k = n$ :
	\begin{equation*} 
	\begin{split}
	\partial \Delta_k x - (-1)^k \Delta_k \partial x\, =\, & 
	(t_{n-1} - s_{n-1}) x \otimes x \ + \ (-1)^{n} \, x \otimes (t_{n-1} - s_{n-1}) x \\ =\, &
	\big( t_{k-1} x \otimes x\, - (-1)^{n} \, x \otimes s_{k-1} x \big) \\ +\, &
	(-1)^n \big( x \otimes t_{k-1} x\, - \, (-1)^n s_{k-1}x \otimes x\big)  \\ =\, &
	\big( 1 + (-1)^n T \big) \big( t_{k-1} x \otimes x\, - (-1)^{n} \, x \otimes s_{k-1} x \big) \\ =\, &
	\big( 1 + (-1)^n T \big) \big( t_{k-1} x \otimes x\, + (-1)^{(n+1)(k-1)} \, x \otimes s_{k-1} x \big) \\ =\, &
	\big( 1 + (-1)^n T \big) \Delta_{k-1}(x).
	\end{split}
	\end{equation*}
	If $0 < k = n-1$ :
	\begin{equation*}
	\begin{split}
	\partial \Delta_k x - (-1)^k \Delta_k \partial x=\,& 
	(t_{k-1} - s_{k-1}) t_k x \otimes x\ + \ (-1)^{(n-1)} t_{k}x \otimes (t_{n-1} - s_{n-1}) x \\ +\, &
	(-1)^{n+1} (t_{n-1} - s_{n-1}) x \otimes s_k x \ - \ x \otimes (t_{k-1} - s_{k-1}) s_k x \\ +\, &
	(-1)^{k+1} \big( t_{k} x \otimes t_{k} x - s_{k} x \otimes s_{k} x \big) \\ =\, &
	\big(t_{k-1} x \otimes x\, + \, x \otimes s_{k-1} x \big) \ - \ \big( x \otimes t_{k-1} x\, + \, s_{k-1}x \otimes x\big)  \\ =\, &
	(1 + (-1)^k T) \big(t_{k-1} x \otimes x\, + \, x \otimes s_{k-1} x \big) \\ =\, &
	(1 + (-1)^k T) \Delta_{k-1}(x).
	\end{split}
	\end{equation*}
	If $0 < k < n-1$ :
	\begin{equation*}
	\begin{split}
	\partial \Delta_k x - (-1)^{k} \Delta_k \partial x=\,& 
	(t_{k-1} - s_{k-1}) t_k x \otimes x \, + \, (-1)^{k} t_k x \otimes (t_{n-1} - s_{n-1}) x \\ +\, &
	(-1)^{(n+1)k} \big( (t_{n-1} - s_{n-1}) x \otimes s_k x\, + \, (-1)^{n} x \otimes (t_{k-1} - s_{k-1}) s_k x \big) \\ +\, &
	(-1)^{k+1} \big(t_k t_{n-1} x \otimes t_{n-1} x \, + \, (-1)^{nk} t_{n-1} x \otimes s_{k} t_{n-1} x \big) \\ +\, &
	(-1)^k \big( t_k s_{n-1} x \otimes s_{n-1} x \, + \, (-1)^{nk} s_{n-1} x \otimes s_{k} s_{n-1} x \big) \\ =\, &
	\big( t_{k-1} x \otimes x\, + (-1)^{(n+1)(k+1)} \, x \otimes s_{k-1} x \big) \\ +\, &
    \big( (-1)^{(n(k+1) + k)} x \otimes t_{k-1} x\, - \, s_{k-1}x \otimes x\big)  \\ =\, &
	(1+(-1)^k T) \big( t_{k-1} x \otimes x\, + \, (-1)^{(n+1)(k-1)} x \otimes s_{k-1} x \big) \\ =\, &	
	(1+(-1)^k T) \Delta_{k-1}(x).
	\end{split}
	\end{equation*}
	Showing that $\varepsilon$ is a counit for $\Delta$ follows from the fact that for any $x \in X_n$
	\begin{equation*}
	\Delta_0 x =
	\begin{cases}
	t_0 x \otimes x + x \otimes s_0 x & n \neq 0 \\ 
	x \otimes x & n = 0
	\end{cases}
	\end{equation*} 
	and $\varepsilon (x') = 1$ for any $x' \in X_0$.
\end{proof}

\begin{lemma} \label{lemma: lift well defined for maps}
	For any morphism $F : X \to Y$ of globular sets, the chain map
	\begin{equation*}
	\chains(F) : \chains(X) \to \chains(Y)
	\end{equation*} 
	is a coalgebra map. 
\end{lemma}

\begin{proof}
	Denote $\chains(F)$ by $f$. Since $F(X_n) \subseteq Y_n$ we have $\varepsilon f = f \varepsilon$ and since $F t_n = t_n F$ and $F s_n = s_n F$ we have
	\begin{equation*}
	\begin{split}
	\big( f \otimes f \big) \Delta_k (x) 
	&= F(t_k x) \otimes F(x) \, + \, (-1)^{(n+1)k} F(x) \otimes F(s_k x) \\ 
	&= t_k F(x) \otimes F(x) \, + \, (-1)^{(n+1)k} F(x) \otimes s_k F(x) \\
	&= \Delta_k f(x)
	\end{split}
	\end{equation*}
	for any $x \in X$ and $k \geq 0$.
\end{proof}

\begin{lemma}
	If $R$ is an integral domain, the function
	\begin{equation*}
	\Set^\Gop(X,Y) \to \coAlg_R \big( \chains(X), \chains(Y) \big)
	\end{equation*}
	is a bijection.
\end{lemma}

\begin{proof}
	Injectivity is immediate. For establishing surjectivity, let us consider $f \in \coAlg_R\big( \chains(X), \chains(Y) \big)$. We will construct $F \in \Set^\Gop(X, Y)$ such that $\chains(F) = f$. From Lemma \ref{lemma: coalgebra maps take group-like to group-like or 0} we know that for any $x \in X$ either $f(x) = y$ for some $y \in Y$ or $f(x) = 0$. Let $x \in X_n$ not in the image of $i_n$. Define
	\begin{equation*}
	F(x) =
	\begin{cases}
	f(x) & f(x) \neq 0 \\
	i_n \big( F(t_{n-1} x) \big) & f(x) = 0.
	\end{cases}
	\end{equation*}
	This recursive definition is well defined because of the augmentation preserving property of $f$. For $x = i_n(y)$ we recursively define $F(x) = i_n F(y)$. We will prove next that $F: X \to Y$ is a map of globular sets.
	
	Let $x \in X_n$ and without loss of generality assume it is not in the image of $i_n$. Let us first assume $f(x) = 0$, then 
	\begin{equation*}
	t_{n-1} F(x) = t_{n-1} i_n F(t_{n-1} x) = F(t_{n-1} x)
	\end{equation*} 
	and
	\begin{equation*}
	s_{n-1} F(x) = s_{n-1} i_n F(t_{n-1} x) = F(t_{n-1} x) \stackrel{?}{=} F(s_{n-1} x).
	\end{equation*} 
	We claim that $F(t_{n-1} x)$ must equal $F(s_{n-1} x)$. Observe that since $f$ is a chain map 
	\begin{equation} \label{equation: ham}
	\boxed{f(x) = 0} \Rightarrow \boxed{f(t_{n-1} x) = f(s_{n-1} x)}
	\end{equation}
	If $f(t_{n-1} x) \neq 0$ or $f(s_{n-1} x) \neq 0$ implication (\ref{equation: ham}) establishes the claims. If $f(t_{n-1} x) = f(s_{n-1} x) = 0$ we have
	\begin{equation*}
	F(t_{n-1} x) = i_{n-1} F( t_{n-2} t_{n-1} x) = i_{n-1} F(t_{n-2} s_{n-1} x) = F(s_{n-1} x).
	\end{equation*}
	Let us now assume $f(x) \neq 0$. For any $0 < k \leq n$ the identity
	\begin{equation*}
	\Delta_{n-k} f(x) = (f \otimes f) \Delta_{n-k} (x)
	\end{equation*} 
	reads
	\begin{equation} \label{equation: jamon}
	\begin{split}
	t_{n-k} F(x) \otimes F(x)\ +\ &(-1)^{(n+1)k} F(x) \otimes s_{n-k} F(x) \\ =\ &f(t_{n-k} x) \otimes F(x)\ +\ (-1)^{(n+1)k} F(x) \otimes f(s_{n-k} x).
	\end{split}
	\end{equation} 
	Therefore, for $0 < k \leq n$
	\begin{equation} \label{equation: queso}
	\begin{split}
	\boxed{f(t_{n-k} x) \neq 0} &\Rightarrow \boxed{t_{n-k} F(x) = F(t_{n-k} x)}  \\
	\boxed{f(s_{n-k} x) \neq 0} &\Rightarrow \boxed{s_{n-k} F(x) = F(s_{n-k} x)\!} 
	\end{split}
	\end{equation}
	Therefore, if both $f(t_{n-1} x) \neq 0$ and $f(s_{n-1} x) \neq 0$ we are done. 
	
	Let us assume $f(t_{n-1} x) = 0$ and notice that $n-1$ must be greater than $0$. It follows from (\ref{equation: jamon}) that $t_{n-1} F(x)$ is in the image of $i_{n-1}$. Writing $t_{n-1} F(x) = i_{n-1} y$ and applying $t_{n-2}$ to this identity gives $t_{n-2}F(x) = y$. Hence,
	\begin{equation*}
	F(t_{n-1}x) \stackrel{\text{def}}{=} i_{n-1} F(t_{n-2}t_{n-1} x) = i_{n-1} F(t_{n-2} x) 
	\stackrel{?}{=} i_{n-1} t_{n-2} F(x) = i_{n-1} y = t_{n-1} F(x).
	\end{equation*}  
	Similarly, when $f(s_{n-1}x) = 0$ we have
	\begin{equation*}
	F(s_{n-1}x) \stackrel{\text{def}}{=} i_{n-1} F(t_{n-2}s_{n-1} x) = i_{n-1} F(t_{n-2} x) 
	\stackrel{?}{=} i_{n-1} t_{n-2} F(x) = s_{n-1} F(x).
	\end{equation*}
	Therefore, we have reduced both claims: $F(t_{n-1} x) = t_{n-1} F(x)$ when $f(t_{n-1} x) = 0$ and $s_{n-1} F(x) = F(s_{n-1} x)$ when $f(s_{n-1} x) = 0$ to showing $F(t_{n-2} x) = t_{n-2} F(x)$. If $f(t_{n-2} x) \neq 0$ then (\ref{equation: queso}) finishes the proof. If not, we repeat the argument and reduce it to $F(t_{n-3} x) = t_{n-3} F(x)$. Because of the augmentation preserving property of $f$ this regression has to end.
\end{proof}

\bibliographystyle{alpha}
\bibliography{globularsets}

\end{document}